\documentclass{ifacconf}

\usepackage{graphicx}
\usepackage{amsmath}
\usepackage{amssymb}
\usepackage{mathrsfs}
\usepackage{epstopdf}
\usepackage{caption}
\usepackage{subcaption}
\usepackage{balance}
\usepackage{natbib}

\newtheorem{theorem}{Theorem}

\newtheorem{proof}{Proof}

\begin{document}
\begin{frontmatter}

\title{Optimal Boundary Control of $2\times2$ Linear Hyperbolic PDEs\thanksref{footnoteinfo}}

\thanks[footnoteinfo]{The present research paper was supported in a part by the EEA Scholarship Program BG09 Project Grant D03-91 under the European Economic Area Financial Mechanism.  This support is greatly appreciated.}

\author[First]{Agus Hasan}, \author[First]{Lars Imsland}, \author[Second]{Ivan Ivanov}, \author[Third]{Snezhana Kostova}, and, \author[Second]{Boryana Bogdanova}
\address[First]{Department of Engineering Cybernetics\\Norwegian University of Science and Technology\\Trondheim, Norway}
\address[Second]{Department of Statistics and Econometrics\\Sofia University - St. Kliment Ohridski\\Sofia, Bulgaria}
\address[Third]{Institute of System Engineering and Robotics\\Bulgarian Academy of Sciences\\Sofia, Bulgaria}

\begin{abstract}
The present paper develops an optimal linear quadratic boundary controller for $2\times2$ linear hyperbolic partial differential equations (PDEs) with actuation on only one end of the domain. First-order necessary conditions for optimality is derived via weak variations and an optimal controller in state-feedback form is presented. The linear quadratic regulator (LQR) controller is calculated from differential algebraic Riccati equations. Numerical examples are performed to show the use of the proposed method.
\end{abstract}

\begin{keyword}
Optimal control, Linear quadratic regulator, Weak variations, Partial differential equations.
\end{keyword}

\end{frontmatter}

\section{Introduction}

Research in control of $2\times2$ linear hyperbolic partial differential equations (PDEs) have attracted great attention in recent years, since they can be used to model many physical systems such as road traffic \citep{Goatin}, gas flow in pipeline \citep{Gugat}, flow of fluids in transmission lines \citep{White,AHasn} and in open channels \citep{Coron}, and mud flow in oil well drilling \citep{ole,land1,haugen,AgusPDE}.

Recent control approaches include backstepping method, where a Volterra integral transformation is used to transform the original system into a stable target system. It turns out that the control law is calculated after solving a Goursat-type PDEs. This type of PDEs can be solved explicitly in terms of Marcum Q-function \citep{Marcum}. Further investigation involved parameter estimations and disturbance rejections have been presented by \cite{Agus1,AHasan,AgusDis} and \cite{aamo}. The backstepping method has been successfully used for control design of many PDEs such as the Schrodinger equation \citep{Guo}, the Ginzburg-Landau equation \citep{Aa}, the Navier-Stokes equation \citep{Vazque}, the surface wave equation \citep{AgusWave}, and the KdV equation \citep{Cerpa1}. Different with other approaches that require the solution of operator Riccati equations, backstepping yields control gain formulas which can be computed using symbolic computation and, in some cases, can even be given explicitly \citep{kris}.

Another control method for the $2\times2$ linear hyperbolic systems of conservative laws has been studied recently by \cite{Van} using receding horizon optimal control approach. Optimal control is a mature concept with numerous applications in science and engineering, e.g., \cite{Dic}. Numerical methods for solving optimal control problems can be divided into two classes of methods. The first method is called the indirect methods, where the calculus of variations is used to determine the first-order optimality conditions of the original optimal control problem \citep{Teo}. The second method is called the direct methods. Here, the state and control variables for the optimal control problem are discretized and the problem is transcribed into nonlinear optimization problems or nonlinear programming problems which can be solved using well-known optimization techniques \citep{AgusOpti}.

In this paper, we are concerned with the problem of linear quadratic optimal control of $2\times2$ linear hyperbolic PDEs using an indirect method. This problem has been previously considered for linear parabolic diffusion-reaction PDEs by \cite{Moura}. Optimal control of these type of PDEs is challenging since in many cases actuation and sensing are often limited to the boundary and the dynamics are notably more complex than ordinary differential equation (ODE) systems. Classical approach usually relies on operator Riccati equation and semi-group theory, which limit the practicality of the method. Thus, by introduce weak variations concept directly to the PDEs, it has been shown by \cite{Mou} that we can bypass semigroup theory and the associated issues with solving operator Riccati equations.

This paper is organized as follow. In section 2, we introduced some definitions and notations used throughout this paper. We stated the problem in section 3. The main result is presented in section 4. Numerical examples are presented in section 5. In this section, we compare the linear quadratic controller with the backstepping controller. Both methods bypass the requirement of solving the operator Riccati equations. The last section contains conclusions.

\section{Preliminaries \& Notations}

To simplify the presentation of this paper, we introduce the following notations \citep{Mou}. For a linear operator $A$ applied to a real-valued function $f$, we define
\begin{eqnarray}
A(f(x)) &:=& \int_0^1\! A(x,y)f(y)\,\mathrm{d}y\nonumber
\end{eqnarray}
where $x\in[0,1]$. The inner product of two functions is defined as
\begin{eqnarray}
\langle f(x),g(x) \rangle &:=& \int_0^1\! f(x)g(x)\,\mathrm{d}x\nonumber
\end{eqnarray}
The subscripts $x$ and $t$ denote partial derivatives with respect to $x$ and $t$, respectively.

\section{Problem Statement}

We consider the following boundary control problem of the 2$\times$2 linear hyperbolic PDEs
\begin{eqnarray}
u_t(x,t) &=& -\epsilon_1u_x(x,t)+c_1(x)v(x,t)\label{main1}\\
v_t(x,t) &=& \epsilon_2v_x(x,t) +c_2(x)u(x,t)\label{main2}\\
u(0,t) &=& qv(0,t)\label{main3}\\
v(1,t) &=& U(t)\label{main4}\\
u(x,0) &=& u_0(x)\label{main5}\\
v(x,0) &=& v_0(x)\label{main6}
\end{eqnarray}
where $x\in[0,1]$, $t>0$, with $\epsilon_1,\epsilon_2>0$. We assume $q\neq0$ and $c_1,c_2\in\mathbb{C}([0,1])$. The initial condition $u_0$ and $v_0$ are assumed to belong to $\mathbb{L}^2([0,1])$.

The objective is to design the state-feedback controller $U(t)$ that minimize the following quadratic function over a finite-time horizon $t\in[0,T]$
\begin{eqnarray}
J&=&\frac{1}{2}\int_0^T\! \left[\langle u(x,t),Q_1(u(x,t))\rangle\right.\nonumber\\
&&\left.+\langle v(x,t),Q_2(v(x,t)) \rangle+RU(t)^2\right]\,\mathrm{d}t\nonumber\\
&&+\frac{1}{2}\langle u(x,T),P_{f1}(u(x,T))\rangle\nonumber\\
&&+\frac{1}{2}\langle v(x,T),P_{f2}(v(x,T))\rangle\label{cost}
\end{eqnarray}
Here $Q_1\geq0$, $Q_2\geq0$, $R>0$, $P_{f1}\geq0$, and $P_{f2}\geq0$ are weighting kernels that respectively weight the states $u$ and $v$, the control $U$, and the terminal states of the system. It is important that $R$ is strictly positive to ensure bounded control signals.

\section{Linear Quadratic Regulator}

In the first part of this section, we derive necessary conditions for optimality of the open-loop linear quadratic control problem using weak variations method. Furthermore, we derive the differential algebraic Riccati equations associated with the state-feedback control law.

\subsection{Open-Loop Control}

The necessary conditions for optimality are utilized using weak variations method presented in \cite{Bern}. This method has been used for reaction-diffusion PDEs by \cite{Mou}.
\begin{theorem}
Consider the $2\times2$ linear hyperbolic PDEs describes by \eqref{main1}-\eqref{main6} defined on the finite-time horizon $t\in[0,T]$ with quadratic cost function \eqref{cost}. Let $u^*(x,t)$, $v^*(x,t)$, $U^*(t)$, $\lambda_1(x,t)$, and $\lambda_2(x,t)$ respectively denote the optimal states, control, and co-states that minimize the quadratic cost. Then the first-order necessary conditions for optimality are given by
\begin{eqnarray}
u^*_t(x,t) &=& -\epsilon_1u^*_x(x,t)+c_1(x)v^*(x,t)\\
v^*_t(x,t) &=& \epsilon_2v^*_x(x,t) +c_2(x)u^*(x,t)\\
-\lambda_{1t}(x,t) &=& \epsilon_1\lambda_{1x}(x,t)+c_2(x)\lambda_2(x,t)+Q_1(u^*(x,t))\label{cos1}\\
-\lambda_{2t}(x,t) &=& -\epsilon_2\lambda_{2x}(x,t)+c_1(x)\lambda_1(x,t)+Q_2(v^*(x,t))\label{cos2}
\end{eqnarray}
with boundary conditions
\begin{eqnarray}
u^*(0,t) &=& qv^*(0,t)\\
v^*(1,t) &=& U^*(t)\\
\lambda_1(1,t) &=& 0\\
\lambda_2(0,t) &=& q\frac{\epsilon_1}{\epsilon_2}\lambda_1(0,t)
\end{eqnarray}
and initial/final conditions
\begin{eqnarray}
u^*(x,0) &=& u_0(x)\\
v^*(x,0) &=& v_0(x)\\
\lambda_1(x,T) &=& P_{f1}(u^*(x,T))\label{cond1}\\
\lambda_2(x,T) &=& P_{f2}(v^*(x,T))\label{cond2}
\end{eqnarray}
where the optimal control input is given by
\begin{eqnarray}
U^*(t) &=& -\frac{\epsilon_2}{R}\lambda_2(1,t)
\end{eqnarray}
\end{theorem}

\begin{proof}
Suppose $u^*(x,t)$, $v^*(x,t)$, and $U^*(t)$ are the optimal states and control input. We introduce the perturbation from the optimal solution as follow
\begin{eqnarray}
u(x,t) &=& u^*(x,t)+\epsilon\delta u(x,t)\nonumber\\
v(x,t) &=& v^*(x,t)+\epsilon\delta v(x,t)\nonumber\\
U(t) &=& U^*(t)+\epsilon\delta U(t)\nonumber
\end{eqnarray}
Substituting these equations into the objective function \eqref{cost}, yields
\begin{eqnarray}
&&J(u^*+\epsilon\delta u,v^*+\epsilon\delta v,U^*+\epsilon\delta U)=\nonumber\\
&&\frac{1}{2}\int_0^T\! \left[\langle u^*+\epsilon\delta u,Q_1(u^*+\epsilon\delta u)\rangle\right.\nonumber\\
&&\left.+\langle v^*+\epsilon\delta v,Q_2(v^*+\epsilon\delta v) \rangle+R(U^*+\epsilon\delta U)^2\right]\,\mathrm{d}t\nonumber\\
&&+\frac{1}{2}\langle u^*(x,T)+\epsilon\delta u(x,T),P_{f1}(u^*(x,T)+\epsilon\delta u(x,T))\rangle\nonumber\\
&&+\frac{1}{2}\langle v^*(x,T)+\epsilon\delta v(x,T),P_{f2}(v^*(x,T)+\epsilon\delta v(x,T))\rangle\nonumber
\end{eqnarray}
Next, we define a Lagrange functional
\begin{eqnarray}
&&g(\epsilon):=\nonumber\\
&&\frac{1}{2}\int_0^T\! \left[\langle u^*+\epsilon\delta u,Q_1(u^*+\epsilon\delta u)\rangle\right.\nonumber\\
&&\left.+\langle v^*+\epsilon\delta v,Q_2(v^*+\epsilon\delta v) \rangle+R(U^*+\epsilon\delta U)^2\right]\,\mathrm{d}t\nonumber\\
&&+\frac{1}{2}\langle u^*(x,T)+\epsilon\delta u(x,T),P_{f1}(u^*(x,T)+\epsilon\delta u(x,T))\rangle\nonumber\\
&&+\frac{1}{2}\langle v^*(x,T)+\epsilon\delta v(x,T),P_{f2}(v^*(x,T)+\epsilon\delta v(x,T))\rangle\nonumber\\
&&+\int_0^T\! \langle \lambda_1(x,t),-\epsilon_1u^*_x-\epsilon_1\epsilon\delta u_x+c_1(x)v^*+c_1(x)\epsilon\delta v\nonumber\\
&&-\frac{\partial}{\partial t}(u^*+\epsilon\delta u) \rangle\,\mathrm{d}t\nonumber\\
&&+\int_0^T\! \langle \lambda_2(x,t),\epsilon_2v^*_x+\epsilon_2\epsilon\delta v_x+c_2(x)u^*+c_2(x)\epsilon\delta u\nonumber\\
&&-\frac{\partial}{\partial t}(v^*+\epsilon\delta v) \rangle\,\mathrm{d}t\nonumber
\end{eqnarray}
where $\lambda_1$ and $\lambda_2$ denote the Lagrange multipliers (co-states). Computing the first derivative of the Lagrange functional with respect to $\epsilon$, yield
\begin{eqnarray}
&&\frac{\mathrm{d}}{\mathrm{d}\epsilon}g(\epsilon)=\nonumber\\
&&\int_0^T\! \left[\langle \delta u,Q_1(u^*+\epsilon\delta u)\rangle+\langle \delta v,Q_2(v^*+\epsilon\delta v) \rangle\right.\nonumber\\
&&\left.+R(U^*+\epsilon\delta U)\delta U\right]\,\mathrm{d}t\nonumber\\
&&+\langle \delta u(x,T),P_{f1}(u^*(x,T)+\epsilon\delta u(x,T))\rangle\nonumber\\
&&+\langle \delta v(x,T),P_{f2}(v^*(x,T)+\epsilon\delta v(x,T))\rangle\nonumber\\
&&+\int_0^T\! \langle \lambda_1(x,t),-\epsilon_1\delta u_x+c_1(x)\delta v-\frac{\partial}{\partial t}(\delta u) \rangle\,\mathrm{d}t\nonumber\\
&&+\int_0^T\! \langle \lambda_2(x,t),\epsilon_2\delta v_x+c_2(x)\delta u-\frac{\partial}{\partial t}(\delta v)\rangle\,\mathrm{d}t\label{bag}
\end{eqnarray}
The inner product from the last two terms at the right hand side can be simplified as follow
\begin{eqnarray}
\langle \lambda_1(x,t),-\epsilon_1\delta u_x\rangle &=& -\epsilon_1\lambda_1(1,t)\delta u(1,t)+\epsilon_1\lambda_1(0,t)q\delta v(0,t)\nonumber\\
&&+\epsilon_1\langle \lambda_{1x},\delta u \rangle\nonumber\\
\langle \lambda_2(x,t),\epsilon_2\delta v_x\rangle &=& \epsilon_2\lambda_2(1,t)\delta U(t)-\epsilon_2\lambda_2(0,t)\delta v(0,t)\nonumber\\
&&-\epsilon_2\langle \lambda_{2x},\delta v \rangle\nonumber
\end{eqnarray}
Furthermore, we compute
\begin{eqnarray}
\int_0^T\! \langle \lambda_1(x,t),\frac{\partial}{\partial t}(\delta u) \rangle\,\mathrm{d}t &=& \langle \lambda_1(x,T),\delta u(x,T) \rangle \nonumber\\
&&- \int_0^T\! \langle \lambda_{1t},\delta u \rangle \,\mathrm{d}t\nonumber\\
\int_0^T\! \langle \lambda_2(x,t),\frac{\partial}{\partial t}(\delta v) \rangle\,\mathrm{d}t &=& \langle \lambda_2(x,T),\delta v(x,T) \rangle \nonumber\\
&&- \int_0^T\! \langle \lambda_{2t},\delta v \rangle \,\mathrm{d}t\nonumber
\end{eqnarray}
Substituting these equations into \eqref{bag}, yields
\begin{eqnarray}
&&\frac{\mathrm{d}}{\mathrm{d}\epsilon}g(\epsilon)=\nonumber\\
&&\int_0^T\! \left[\langle Q_1(u^*+\epsilon\delta u),\delta u\rangle+\langle \epsilon_1\lambda_{1x}+c_2(x)\lambda_2+\lambda_{1t},\delta u \rangle\right]\,\mathrm{d}t\nonumber\\
&&+\int_0^T\! \left[\langle Q_2(v^*+\epsilon\delta v),\delta v \rangle\right.\nonumber\\
&&\left.+\langle -\epsilon_2\lambda_{2x}+c_1(x)\lambda_1+\lambda_{2t},\delta v \rangle\right]\,\mathrm{d}t\nonumber\\
&&+\int_0^T\! \left[ R(U^*+\epsilon\delta U)+\epsilon_2\lambda_2(1,t)\right]\delta U\,\mathrm{d}t\nonumber\\
&&+\int_0^T\! \left[ -\epsilon_1\lambda_1(1,t)\delta u(1,t)+\epsilon_1\lambda_1(0,t)q\delta v(0,t)\right.\nonumber\\
&&\left.-\epsilon_2\lambda_2(0,t)\delta v(0,t)\right]\,\mathrm{d}t\nonumber\\
&&+\langle P_{f1}(u^*(x,T)+\epsilon\delta u(x,T))-\lambda_1(x,T),\delta u(x,T)\rangle\nonumber\\
&&+\langle P_{f2}(v^*(x,T)+\epsilon\delta v(x,T))-\lambda_2(x,T),\delta v(x,T)\rangle\nonumber
\end{eqnarray}
If we evaluate the above equation at $\epsilon=0$, we have
\begin{eqnarray}
&&\left.\frac{\mathrm{d}}{\mathrm{d}\epsilon}g(\epsilon)\right|_{\epsilon=0}=\nonumber\\
&&\int_0^T\! \left[\langle Q_1(u^*),\delta u\rangle+\langle \epsilon_1\lambda_{1x}+c_2(x)\lambda_2+\lambda_{1t},\delta u \rangle\right]\,\mathrm{d}t\nonumber\\
&&+\int_0^T\! \left[\langle Q_2(v^*),\delta v \rangle+\langle -\epsilon_2\lambda_{2x}+c_1(x)\lambda_1+\lambda_{2t},\delta v \rangle\right]\,\mathrm{d}t\nonumber\\
&&+\int_0^T\! \left[ RU^*+\epsilon_2\lambda_2(1,t)\right]\delta U\,\mathrm{d}t\nonumber\\
&&+\int_0^T\! \left[ -\epsilon_1\lambda_1(1,t)\delta u(1,t)+\epsilon_1\lambda_1(0,t)q\delta v(0,t)\right.\nonumber\\
&&\left.-\epsilon_2\lambda_2(0,t)\delta v(0,t)\right]\,\mathrm{d}t\nonumber\\
&&+\langle P_{f1}(u^*(x,T))-\lambda_1(x,T),\delta u(x,T)\rangle\nonumber\\
&&+\langle P_{f2}(v^*(x,T))-\lambda_2(x,T),\delta v(x,T)\rangle\nonumber
\end{eqnarray}
The necessary conditions for optimality is when
\begin{eqnarray}
\left.\frac{\mathrm{d}}{\mathrm{d}\epsilon}g(\epsilon)\right|_{\epsilon=0}=0\nonumber
\end{eqnarray}
Thus, we have the following first-order optimality conditions
\begin{eqnarray}
-\lambda_{1t}(x,t) &=& \epsilon_1\lambda_{1x}(x,t)+c_2(x)\lambda_2(x,t)+Q_1(u^*(x,t))\nonumber\\
-\lambda_{2t}(x,t) &=& -\epsilon_2\lambda_{2x}(x,t)+c_1(x)\lambda_1(x,t)+Q_2(v^*(x,t))\nonumber\\
\lambda_1(1,t) &=& 0\nonumber\\
\lambda_2(0,t) &=& q\frac{\epsilon_1}{\epsilon_2}\lambda_1(0,t)\nonumber\\
\lambda_1(x,T) &=& P_{f1}(u^*(x,T))\nonumber\\
\lambda_2(x,T) &=& P_{f2}(v^*(x,T))\nonumber\\
U^*(t) &=& -\frac{\epsilon_2}{R}\lambda_2(1,t)\nonumber
\end{eqnarray}
This concludes the proof.
\end{proof}

\subsection{State-Feedback Control}

Utilizing the results by \cite{Moura}, for the state-feedback problem, first we postulate the co-states $\lambda_1$ and $\lambda_2$ are related with the states $u$ and $v$, respectively, according to the following transformation
\begin{eqnarray}
\lambda_1(x,t) = P^{t1}(u^*(x,t)) &=& \int_0^1\! P_1(x,y,t)u^*(y,t)\,\mathrm{d}y\label{lam1}\\
\lambda_2(x,t) = P^{t2}(v^*(x,t)) &=& \int_0^1\! P_2(x,y,t)v^*(y,t)\,\mathrm{d}y\label{lam2}
\end{eqnarray}
The superscript on $P^{t}$ suggest the linear operator is time dependent. For the linear $2\times2$ hyperbolic PDEs, we have the following result.
\begin{theorem}
The optimal control in state-feedback form is given by
\begin{eqnarray}
U^*(t) &=& -\frac{\epsilon_2}{R}\int_0^1\! P_2(1,y,t)v^*(y,t)\,\mathrm{d}y\label{optcon}
\end{eqnarray}
where the time varying transformation $P_2$ is the solution to the following differential algebraic Riccati equations
\begin{eqnarray}
-P_{2t} &=& -\frac{\epsilon_2^2}{R}P_2(x,1,t)P_2(1,y,t)\nonumber\\
&&-\epsilon_2P_{2y}-\epsilon_2P_{2x}+Q_2\label{11}\\
-P_{1t} &=& \epsilon_1P_{1y}+\epsilon_1P_{1x}+Q_1\\
c_1(x)P_1 + c_2(x)P_2 &=& 0\label{111}
\end{eqnarray}
with boundary conditions
\begin{eqnarray}
P_2(x,0,t) &=& 0\label{21}\\
P_1(x,1,t) &=& 0\\
P_1(x,0,t) &=& 0\\
P_1(1,y,t) &=& 0\\
P_1(0,y,t) &=& 0\\
P_2(0,y,t) &=& 0\label{22}
\end{eqnarray}
and final conditions
\begin{eqnarray}
P_1(x,y,T) &=& P_{f1}(x,y)\\
P_2(x,y,T) &=& P_{f2}(x,y)\label{fin2}
\end{eqnarray}
\end{theorem}

\begin{proof}
Evaluating every term in \eqref{cos1}-\eqref{cos2} using the definitions \eqref{lam1}-\eqref{lam2} and integration by parts, we have the first terms
\begin{eqnarray}
\lambda_{1t}(x,t) &=& \int_0^1\! P_{1t}(x,y,t)u^*(y,t)\,\mathrm{d}y\nonumber\\
&&-\epsilon_1P_1(x,1,t)u^*(1,t)+\epsilon_1P_1(x,0,t)u^*(0,t)\nonumber\\
&&+\int_0^1\! \epsilon_1P_{1y}(x,y,t)u^*(y,t)\,\mathrm{d}y\nonumber\\
&&+ \int_0^1\! c_1(x)P_1(x,y,t)v^*(y,t)\,\mathrm{d}y\nonumber\\
\lambda_{2t}(x,t) &=& \int_0^1\! P_{2t}(x,y,t)v^*(y,t)\,\mathrm{d}y\nonumber\\
&&+\epsilon_2P_2(x,1,t)v^*(1,t)-\epsilon_2P_2(x,0,t)v^*(0,t)\nonumber\\
&&-\int_0^1\! \epsilon_2P_{2y}(x,y,t)v^*(y,t)\,\mathrm{d}y\nonumber\\
&&+\int_0^1\! c_2(x)P_2(x,y,t)u^*(y,t)\,\mathrm{d}y\nonumber
\end{eqnarray}
the second terms are given by
\begin{eqnarray}
\epsilon_1\lambda_{1x}(x,t) &=& \int_0^1\! \epsilon_1P_{1x}(x,y,t)u^*(y,t)\,\mathrm{d}y\nonumber\\
-\epsilon_2\lambda_{2x}(x,t) &=& -\int_0^1\! \epsilon_2P_{2x}(x,y,t)v^*(y,t)\,\mathrm{d}y\nonumber
\end{eqnarray}
the third terms are given by
\begin{eqnarray}
c_2(x)\lambda_2(x,t) &=& \int_0^1\! c_2(x)P_2(x,y,t)v^*(y,t)\,\mathrm{d}y\nonumber\\
c_1(x)\lambda_1(x,t) &=& \int_0^1\! c_1(x)P_1(x,y,t)u^*(y,t)\,\mathrm{d}y\nonumber
\end{eqnarray}
and the last terms are given by
\begin{eqnarray}
Q_1(u^*(x,t)) &=& \int_0^1\! Q_1(x,y)u^*(y,t)\,\mathrm{d}y\nonumber\\
Q_2(v^*(x,t)) &=& \int_0^1\! Q_2(x,y)v^*(y,t)\,\mathrm{d}y\nonumber
\end{eqnarray}
Plugging these terms into \eqref{cos1}-\eqref{cos2}, we have \eqref{11}-\eqref{22}. The last two conditions are obtained from \eqref{cond1}-\eqref{cond2}. This concludes the proof.
\end{proof}

Note that for the infinite-time horizon, the LQR controller is given by the following steady-state solution of the differential algebraic Riccati equations
\begin{eqnarray}
-\frac{\epsilon_2^2}{R}P_2^{\infty}(x,1,t)P_2^{\infty}(1,y,t)-\epsilon_2P_{2y}^{\infty}-\epsilon_2P_{2x}^{\infty}&=& -Q_2\nonumber\\
\epsilon_1P_{1y}^{\infty}+\epsilon_1P_{1x}^{\infty}&=& -Q_1\nonumber\\
c_1(x)P_1^{\infty} + c_2(x)P_2^{\infty} &=& 0\nonumber
\end{eqnarray}
with boundary conditions \eqref{21}-\eqref{22}. The time-invariant steady-feedback controller is given by
\begin{eqnarray}
U^*(t) &=& -\frac{\epsilon_2}{R}\int_0^1\! P_2^{\infty}(1,y,t)v^*(y,t)\,\mathrm{d}y\nonumber
\end{eqnarray}
Unfortunately, we cannot prove the existence of the solution satisfy \eqref{11}-\eqref{fin2}. For the example in the following section, we use parameters and weighting kernels such that the solution exists. Thus, the well-posedness problem of the differential algebraic Riccati equations remains open.

\section{Numerical Simulations}

In this section, we perform two cases. In the first case, we show the LQR controller solve the optimal control problem in finite time. In the second case, the LQR controller is compared with the backstepping controller. In both cases, we use the following parameters
\begin{center}
\begin{tabular}{lr}
\hline
\multicolumn{2}{c}{Simulation parameters} \\
\cline{1-2}
Symbol & Value \\
\hline
$\epsilon_1$    & 1      \\
$\epsilon_2$    & 1      \\
$c_1$     & 10      \\
$c_2$     & 20      \\
q      & 1       \\
T      & 1       \\
\hline
\end{tabular}
\end{center}

\subsection{Case 1: Performance of the LQR controller}

In the first case, we choose the following weighting kernels
\begin{center}
\begin{tabular}{lr}
\hline
\multicolumn{2}{c}{Simulation kernels} \\
\cline{1-2}
Symbol & Value \\
\hline
R      & 1       \\
$P_{f1}$      & $\sin\left(\pi x\right)\sin\left(\pi y\right)$      \\
$P_{f2}$      & $5\sin\left(\pi x\right)\sin\left(\pi y\right)$      \\
$Q_1$      & $10\sin\left(\pi x\right)\sin\left(\pi y\right)$      \\
$Q_2$      & $20\sin\left(\pi x\right)\sin\left(\pi y\right)$       \\
\hline
\end{tabular}
\end{center}
The algebraic Riccati equation is given by
\begin{eqnarray}
-P_{2t} &=& -P_2(x,1,t)P_2(1,y,t)-P_{2y}-P_{2x}+Q_2\nonumber\\
-P_{1t} &=& P_{1y}+P_{1x}+Q_1\nonumber\\
P_1 + 2P_2 &=& 0\nonumber
\end{eqnarray}
The above equation is solved numerically. The solution is used to calculate the optimal state-feedback controller \eqref{optcon} and the result can be seen in the following figure.
\begin{figure}[h!]
  \centering
      \includegraphics[width=0.5\textwidth]{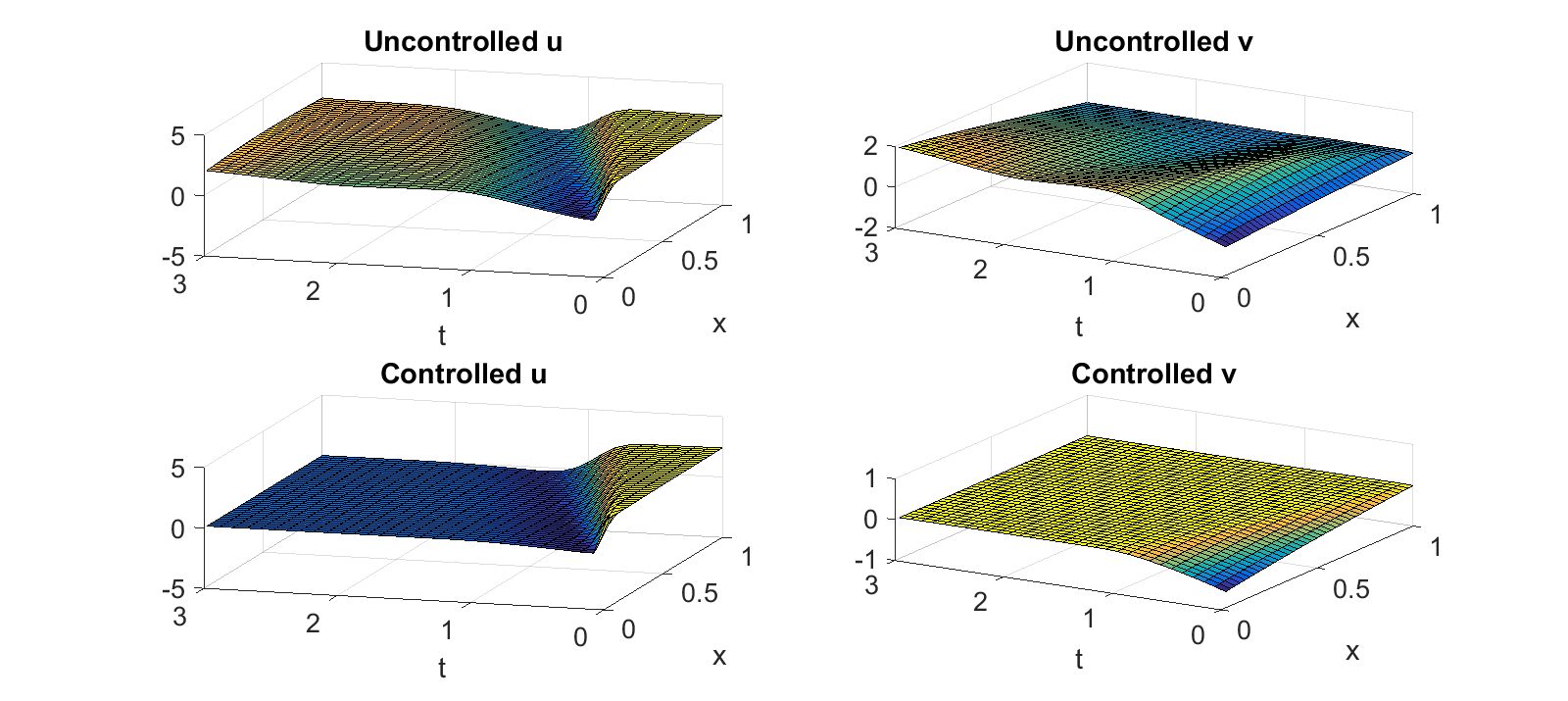}
  \caption{Uncontrolled vs controlled case using the LQR controller \eqref{optcon}.}
\label{fig1}
\end{figure}

It can be observed from figure 1 that the LQR controller \eqref{optcon} regulates the linear hyperbolic system into its equilibrium.

\subsection{Case 2: comparison with the backstepping controller}

We compare the controller \eqref{optcon} with the controller obtained from the backstepping method by \cite{vaz}. The controller from the backstepping method is given by
\begin{eqnarray}
U(t) &=& \int_0^1\! \left(K^{vu}(1,y)u(y,t) + K^{vv}(1,y)v(y,t)\right) \,\mathrm{d}y\nonumber
\end{eqnarray}
where the backstepping kernels $K^{vu}$ and $K^{vv}$ are obtained from
\begin{eqnarray}
\epsilon_2K_x^{vu}(x,y)-\epsilon_1K_y^{vu}(x,y) &=& c_2(y)K^{vv}(x,y)\nonumber\\
\epsilon_2K_x^{vv}(x,y)+\epsilon_2K_y^{vv}(x,y) &=& c_1(y)K^{vu}(x,y)\nonumber
\end{eqnarray}
on a triangular domain $\{(x,y):0\leq y\leq x\leq1\}$ with the following boundary conditions
\begin{eqnarray}
K^{vv}(x,0) &=& \frac{q\epsilon_1}{\epsilon_2}K^{vu}(x,0)\nonumber\\
K^{vu}(x,x) &=& -\frac{c_2(x)}{\epsilon_1+\epsilon_2}\nonumber
\end{eqnarray}
The solutions of the backstepping kernels are given by
\begin{eqnarray}
K^{vu}(1,y) &=& -\frac{1}{2}\left\{10I_0\left[10\sqrt{\frac{1-y}{1+y}}\right]\right.\nonumber\\
&&\left.+10\sqrt{\frac{1-y}{1+y}}I_1\left[10\sqrt{\frac{1-y}{1+y}}\right]\right\}\nonumber\\
K^{vv}(1,y) &=& -\frac{1}{2}\left\{10I_0\left[10\sqrt{\frac{1-y}{1+y}}\right]\right.\nonumber\\
&&\left.+10\sqrt{\frac{1-y}{1+y}}I_1\left[10\sqrt{\frac{1-y}{1+y}}\right]\right\}\nonumber
\end{eqnarray}
where $I_0$ and $I_1$ denote modified Bessel function defined as
\begin{eqnarray}
I_n(x) &=& \sum_{m=0}^{\infty} \frac{(\frac{x}{2})^{n+2m}}{m!(m+n)!}\nonumber
\end{eqnarray}
We calculate the control signal and the tracking error function $\|u\|_{\mathbb{L}^2}$. In figure 2, we can observe that the LQR controller performs better compared to the backstepping controller. However, while the gains in the backstepping controller can be computed analytically, the existence of the solution for the algebraic Riccati equation remains an open question.

\begin{figure}[h!]
  \centering
      \includegraphics[width=0.5\textwidth]{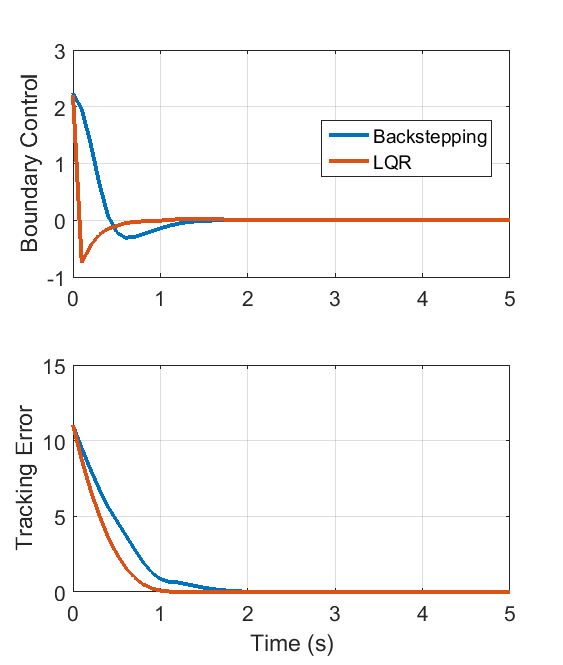}
  \caption{Comparison between backstepping and LQR controller.}
\label{fig1}
\end{figure}

\section{Conclusions}

In this paper, we derived LQR results for boundary controlled $2\times2$ linear hyperbolic PDEs. The necessary conditions for optimality for the open-loop system are obtained via weak variations, while the explicit state-feedback is derived from the co-state systems. In the given examples, the state-feedback laws are calculated after solving the differential algebraic Riccati equations for special cases. For general cases, the existence of solution remains an open question. This will be considered in the future work.

\balance
\bibliography{ifacconf}

\end{document}